\documentclass[11pt]{amsart}

\usepackage[colorlinks=true, linkcolor=blue, anchorcolor=blue, citecolor=blue, filecolor=blue, menucolor= blue, urlcolor=red]{hyperref}

\usepackage{amsmath,amssymb,amsthm}
\usepackage{verbatim}
\usepackage{comment}

\long\def\eatit#1{}
\newtheorem{thm}{Theorem}[section]
\newtheorem{prop}[thm]{Proposition}
\newtheorem{lem}[thm]{Lemma}
\newtheorem{cor}[thm]{Corollary}

\theoremstyle{definition}

\newtheorem{rem}[thm]{Remark}

\newcommand{\pr}[1]{{{\bf P}^{#1}}}

\newcommand{\field}{K}

\begin{document}
\title{Containment Counterexamples for ideals of various configurations of points in $\pr N$}

\author{Brian Harbourne}
\address{Department of Mathematics\\
University of Nebraska\\
Lincoln, NE 68588-0130 USA}
\email{bharbour@math.unl.edu}
\author{Alexandra Seceleanu}
\address{Department of Mathematics\\
University of Nebraska\\
Lincoln, NE 68588-0130 USA}
\email{aseceleanu@unl.edu}

\thanks{The first author's work on this project
was sponsored by the National Security Agency under Grant/Cooperative
agreement ``Topics in Algebra and Geometry, Curves and Containments'' Number H98230-13-1-0213.
The United States Government is authorized to reproduce and distribute reprints
notwithstanding any copyright notice. }

\begin{abstract}
When $I$ is the radical homogeneous ideal of a finite set of points in 
projective $N$-space, $\pr{N}$, over a field $\field$, it has been conjectured
that $I^{(rN-N+1)}$ should be contained in $I^r$ for all $r\geq 1$. Recent counterexamples
show that this can fail when $N=r=2$. We study properties of the resulting ideals.
We also show that failures occur for 
infinitely many $r$ in every characteristic $p>2$ when $N=2$,
and we find additional positive characteristic failures when $N>2$.
\end{abstract}

\date{June 10, 2013}



\maketitle

\section{Introduction}

We work over an arbitrary field $\field$.
Let $p_1,\ldots,p_n\in\pr N$ and let $I(p_i)\subset \field[\pr N]=\field[x_0,\ldots,x_N]$ 
be the ideal generated by all homogeneous polynomials vanishing on the point $p_i$.
Thus $I=\cap_iI(p_i)$ is the ideal generated by all homogeneous polynomials vanishing 
on each point $p_i$ and $I^{(m)}= \cap_i(I(p_i)^m)$ is the $m$-th symbolic power of the ideal.
Equivalently, $I^{(m)}$ is the saturation of $I^m$ with respect to the ideal $M=(x_0,\ldots,x_N)$.

In \cite{refELS,refHoHu} it is shown that $J^{(rN)}\subseteq J^r$ for all homogeneous ideals
$J\subseteq\field[\pr N]$ for all $r\geq 1$, or, equivalently, that $J^{(m)}\subseteq J^r$ 
for all homogeneous ideals 
$J\subseteq\field[\pr N]$ whenever $m/r \geq N$. This raises the question:
what is the smallest real $c$ such that $J^{(m)}\subseteq J^r$ for all homogeneous ideals 
whenever $m/r \geq c$? It is shown in \cite{refBH} that the smallest such $c$ is $c=N$.
In that sense, the result of \cite{refELS,refHoHu} is optimal. But while the containment
$J^{(rN)}\subseteq J^r$ cannot be made tighter by replacing $N$ by a smaller
coefficient, one can ask whether there might be an integer $c>0$ such that
$J^{(rN-c)}\subseteq J^r$ always holds. As a specific problem, 
in the case of any radical ideal $I\subseteq \field[\pr2]$ defining a finite set of points of $\pr2$, 
Craig Huneke asked if $I^{(3)}\subseteq I^2$ always holds. 
After exploring examples, the first author conjectured \cite{refB. et al}
that $I^{(Nr-N+1)}\subseteq I^r$ should hold for all $r\geq 1$ and all $N\geq 1$. 
(For example it is known that $I^{(Nr-N+1)}\subseteq I^r$ holds when $r$ is a power of 
$\operatorname{char}(\field)$ and when $I$ is a monomial ideal \cite{refB. et al}, and in many other 
cases \cite{refBH, refBH2, refBCH, refDJ, refDu, refHaHu}.)

One can also make the containment $J^{(rN)}\subseteq J^r$ tighter by replacing the ideal $J^r$
by something smaller. This led the authors of \cite{refHaHu} to explore the possibility of containments of the form
$J^{(rN)}\subseteq M^iJ^r$ where $M=(x_0,\ldots,x_N)$.
They propose several additional conjectures along these lines.
For example, for any radical ideal $I\subseteq \field[\pr{N}]$ defining a finite set of points of $\pr{N}$,
they conjecture that $I^{(rN)}\subseteq M^{r(N-1)}I^r$ and $I^{(rN-N+1)}\subseteq M^{(r-1)(N-1)}I^r$.

The conjecture $I^{(rN)}\subseteq M^{r(N-1)}I^r$ is still open.
The first indication that the conjecture $I^{(rN-N+1)}\subseteq M^{(r-1)(N-1)}I^r$ 
was overly optimistic arose with the example
$$I=(x_0^2x_1,x_1^2x_2,x_2^2x_0,x_0x_1x_2)\subset\field[\pr2]$$ 
\cite{refEHS} (announced in a talk in October 2012 by A.\ Hoefel). 
For this example we have $I^{(rN-N+1)}\not\subseteq M^{(r-1)(N-1)}I^r$
for $N=r=2$. (It should be noted that although $I$ in this case is not a radical ideal
and thus not a counterexample to the precise statement of the conjecture, 
$I$ is the ideal of six points of $\pr2$, these six points being the three coordinate vertices
and three infinitely near points corresponding to tangent directions pointing along the coordinate axes,
and thus $I$ is quite close to being a counterexample to the spirit of the conjecture.)

More recently counterexamples with $N=r=2$ have arisen over the complex numbers $\field={\bf C}$
for the conjecture $I^{(Nr-N+1)}\subseteq I^r$ \cite{refDST}, thus answering Huneke's original question negatively.
The main counterexample of \cite{refDST} is $I=(x_0(x_1^3-x_2^3),x_1(x_0^3-x_2^3),x_2(x_0^3-x_1^3))$
(they remark that the exponents of 3 can be replaced by any integer $m\geq 3$). This ideal
is the ideal of 12 points lying at the points of pair-wise intersections of 9 lines
arranged so each point lies on three lines and each line passes through 4 points. 
This counterexample inspired another similar counterexample given in \cite{refBCH},
comprising any 12 of the 13 ${\bf Z}/3{\bf Z}$-points in $\pr2$ over any field $\field$ of characteristic 3.

In this paper we study these initial counterexamples to $I^{(Nr-N+1)}\subseteq I^r$
and we provide additional counterexamples for various values of $N$ and $r$
but all in positive characteristic. Our results raise a number of questions.
For example, are the counterexamples for $r>2$ or for $N>2$ purely a 
positive characteristic phenomenon? Even if this turns out to be the case,
it is at least clear that there is no universal containment statement of the form
$I^{(Nr-N+1)}\subseteq I^r$. But having $I^{(Nr-c)}\subseteq I^r$ for $c=N-1$
was the most optimistic possibility
(in the sense that there are easy counterexamples to $I^{(Nr-N)}\subseteq I^r$).
The most conservative possibility would be to take $c=1$ with $N>2$;
i.e., it still may be true that $I^{(Nr-1)}\subseteq I^r$ holds for all radical
ideals $I$ of finite sets of points in $\pr{N}$ for all $r\geq 1$ as long as $N>2$.
An updated version of Huneke's question would thus be:
is it always true for the ideal $I$ of a finite set of points in $\pr3$ that
$I^{(5)}\subseteq I^2$?

This paper is structured as follows. In section \ref{ressect} we show that
the resurgences of the counterexample ideals of \cite{refDST} and \cite{refBCH}
are different, and thus, despite their combinatorial similarities, they are
significantly different algebraically.
In section \ref{addcntrex} we obtain counterexamples over $\pr2$
to $I^{(2r-1)}\subseteq I^r$ for infinitely many $r$.
In section \ref{N>=2} we provide counterexamples
to $I^{(Nr-N+1)}\subseteq I^r$ for various values of $N$ and $r$
for various characteristics $p>0$.

\subsection{More details}
The ideal $I$ for the counterexample of \cite{refDST} is 
$I=(x_0(x_1^3-x_2^3),x_1(x_0^3-x_2^3),x_2(x_0^3-x_1^3))$ with $\field={\bf C}$. 
The zero locus of $I$ consists of the
three coordinate vertices, $[1:0:0],[0:1:0]$ and $[0:0:1]$, plus the 9 point complete intersection defined by
$x_0^3-x_1^3$ and $x_1^3-x_2^3$. Note that the cubic pencil defined by 
$x_0^3-x_1^3$ and $x_1^3-x_2^3$ also contains the form $x_0^3-x_2^3$.
Indeed these are precisely the singular members of this pencil. 
The form $F=(x_1^3-x_2^3)(x_0^3-x_1^3)(x_0^3-x_1^3)$ is clearly in $I^{(3)}$ since it defines
9 lines, some three of which go through each of the 12 points,
but it can be shown that $F\not\in I^2$, hence $I^{(3)}\not\subseteq I^2$.

The example given in \cite{refBCH} was modeled after that of \cite{refDST}.
Let $\operatorname{char}(\field)=3$.
Let $I$ be the radical ideal defining any 12 of the 13 ${\bf Z}/3{\bf Z}$-points in $\pr2(\field)$.
Let $q$ be the excluded 13th point (we may assume $q=[1:0:0]$). There are 13 lines defined over ${\bf Z}/3{\bf Z}$, 
exactly 9 of which do not contain $q$. There are 4 ${\bf Z}/3{\bf Z}$-points on each of the 9 lines,
and each of the 12 points lies on exactly 3 of the 9 lines. Let $F$ be the product of the forms defining
these 9 lines, so $F=F_1F_2F_3$, where
$F_1=x_0(x_0^2-x_1^2)$ is the product of the three lines through the point $[0:0:1]$ which do not contain $q$,
$F_2=(x_0-x_2)((x_0-x_2)^2-x_1^2)$ is the product of the three lines through the 
point $[1:0:1]$ which do not contain $q$, and 
$F_3=(x_0+x_2)((x_0+x_2)^2-x_1^2)$ is the product of the three lines through the 
point $[-1:0:1]$ which do not contain $q$. Note that the cubic pencil defined by $F_1$ and $F_2$ contains $F_3$;
the zero locus of $I$ consists of the three points on $x_1=0$ which are the singular points
of $F_1$, $F_2$ and $F_3$, together with the 9 point complete intersection defined by $F_1$ and $F_2$.
Then $F$ is again clearly in $I^{(3)}$, but it turns out that
$I=(x_0x_1(x_0^2-x_1^2), x_0x_2(x_0^2-x_2^2), x_1x_2(x_1^2-x_2^2), x_0(x_0^2-x_1^2)(x_0^2-x_2^2))$.
Note that the cubic generators of $I$ each vanish on all 13 of the ${\bf Z}/3{\bf Z}$-points; it is only
the quintic generator that avoids $q$.
Thus the least degree of an element of $I$
which does not vanish at $q$ is 5, so the least degree of an element of $I^2$
which does not vanish at $q$ is 10, hence $F\not\in I^2$, so again $I^{(3)}\not\subseteq I^2$.

Despite the similarity of the example of \cite{refDST} with that of \cite{refBCH},
there are some notable differences. Both are based on cubic pencils having three 
singular members each of which is a union of three concurrent lines, but the singular members
for the former have their singular points at the coordinate vertices while the singular members for
the latter have colinear singular points (and the pencil is quasi-elliptic, meaning that every member of
the pencil is singular). Moreover, as we will see the 
resurgence is different for the two cases (we recall that the resurgence is 
$\rho(I)=\sup\{\frac{m}{r}:I^{(m)}\not\subseteq I^r\}$ \cite{refBH}). For the example in \cite{refDST} we have only
the upper bound $\rho(I)\leq 1.6$ but for the example of \cite{refBCH} we have the exact value
$\rho(I)=\frac{5}{3}$.

Over $\field={\bf C}$ 
the only additional counterexamples to $I^{(Nr-N+1)}\subseteq I^r$ 
known are for $N=r=2$, specifically $m^2+3$ points in the plane having ideal
$I=(x_0(x_1^m-x_2^m),x_1(x_0^m-x_1^m),x_2(x_0^m-x_1^m))$ for $m\geq 3$. 
The ideal $I$ with the same generators but taking 
$\operatorname{char}(\field)=p>2$ such that $\field$
contains a finite field of order $p^s$ for some $s\geq 1$ with $p^s -1$
divisible by $m$ also has $I^{(3)}\not\subseteq I^2$; 
see Proposition \ref{moreexs}.

In positive characteristics we have many additional examples, and not just for $N=2$, but all involve
taking all but one of the $\field'$-points of $\pr{N}(\field')$, where $\field'$ is a finite field contained in $\field$. 
If we denote the excluded point by $q$,
the method in each case is to find a degree $t$ such that $q$ is not in the zero locus of
$(I^{(Nr-N+1}))_t$ but it is in the zero locus of $(I^r)_t$ (where, given a homogeneous ideal $J$,
$J_t$ denotes the homogeneous component of $J$ of degree $t$), from which it immediately follows that
$(I^{(Nr-N+1}))_t\not\subseteq (I^r)_t$ and hence that $I^{(Nr-N+1})\not\subseteq I^r$.
For $N=2$, for example, we have $(I^{(2r-1)})_t\not\subseteq (I^r)_t$ for $r=(p^s+1)/2$ 
and $t=p^{2s}$ when $I$ is the ideal of all of the $\field'$-points except $q=[0:1:0]$ in $\pr2(\field')$ 
for a field $\field'$ with $|\field'|=p^s$ (see Proposition \ref{counterexP2charp}). 
We also have additional examples of $I^{(Nr-N+1)}\not \subseteq I^r$
taking all but one of the $\field'$-points in $\pr{N}(\field')$ where $|\field'|=p$,
some with $r=2$ and $N=(p+1)/2$, and others with $r=(p+N-1)/N$ for $N\geq 2$; see section \ref{N>=2}. 

We will use the following notation. Given any homogeneous ideal $(0)\neq J\subseteq\field[\pr N]$,
$\alpha(J)$ is the least degree among nonzero elements of $J$, and if $N=2$ and $J$ defines
a finite set of points, then $\beta(J)$ is the least degree $t$ such that the zero locus of 
$J_t$ is 0-dimensional (i.e., does not have a nonconstant common factor).

\section{Resurgences}\label{ressect}

The combinatorial structure of the arrangement of the points for the counterexample
to $I^{(3)}\subseteq I^2$ given in \cite{refDST} is the same as the characteristic 3 counterexample given in
\cite{refBCH}: both consist of 12 points occurring at the points of intersection of 9 lines,
with four points on each line and three lines through each point.
Algebraically, however, these example are different. For example, the ideal of the 12 points for
\cite{refDST} is generated in degree 4, but for \cite{refBCH} there is in addition a degree 5 generator.
This difference leads also to differences in the resurgence, as we now show.
First we consider the counterexample 
given in \cite{refDST}. 

\begin{lem}\label{HaHuProp3.5}
Let $I=(x_0(x_1^3-x_2^3),x_1(x_0^3-x_2^3),x_2(x_0^3-x_1^3))\subseteq{\bf C}[x_0,x_1,x_2]$. Then 
$I^{(3t)}=(I^{(3)})^t$ for all $t\geq 1$.
\end{lem}

\begin{proof}
It is easy to check that this is the radical ideal of the $n=12$ distinct points in $\pr2({\bf C})$
mentioned above. Using {\em Macaulay2} \cite{refM2}, for example, it is easy to check that $\alpha(I^{(3)})=9$
and that $I^{(3)}$ has a unique generator in degree 9, with all of 
its other minimal homogeneous generators being in
degree 12. Thus $\beta(I^{(3)})=12$, so $\alpha(I^{(m)})\beta(I^{(m)})=m^2n$ for $m=3$.
Thus by \cite[Proposition 3.5]{refHaHu} we have $I^{(3t)}=(I^{(3)})^t$ for all $t\geq 1$.
\end{proof}

Lemma \ref{HaHuProp3.5} allows us to apply a computational method (see \cite{refD}) for 
obtaining upper bounds on $\rho(I)$ whenever powers of a certain symbolic 
power are symbolic. In this case, it allows us to
compute $\rho(I)$ to any desired precision.
The method uses the following lemma.

\begin{lem}\label{ADcompMeth}
Let $I=(x_0(x_1^3-x_2^3),x_1(x_0^3-x_2^3),x_2(x_0^3-x_1^3))\subseteq{\bf C}[x_0,x_1,x_2]$. 
Let $r$ be a positive integer, let $i$ be the remainder when dividing $r$ by 4, 
so $r=4t+i$ for some $t\geq0$, and let $j=0$ if $i=0$; otherwise let $j=4-i$.
Then $\frac{m}{r}\geq \frac{6}{4}(1+\frac{j}{r})$ implies $I^{(m)}\subseteq I^r$.
\end{lem}

\begin{proof}
Using Macaulay2 we can check that $I^{(6)}\subseteq I^4$.
If $i=0$, then $j=0$ and so $m\geq \frac{6}{4}r\geq6t$, hence $I^{(m)}\subseteq I^{(6t)}$, but $I^{(6t)}=(I^{(6)})^t$
by Lemma \ref{HaHuProp3.5}, so $I^{(m)}\subseteq (I^{(6)})^t\subseteq I^{4t}=I^r$.
Now assume $0<i\leq 3$, so $j=4-i$.
Since $\frac{m}{r}\geq \frac{6}{4}(1+\frac{4-i}{r})$, we have $m\geq \frac{6}{4}(r+4-i)=6(t+1)$
and $r<4(t+1)$.
Thus $I^{(m)}\subseteq (I^{(6)})^{t+1}\subseteq I^{4(t+1)}\subseteq I^r$.
\end{proof}

\begin{cor}\label{rhoKrakowExample}
Let $I=(x_0(x_1^3-x_2^3),x_1(x_0^3-x_2^3),x_2(x_0^3-x_1^3))\subseteq{\bf C}[x_0,x_1,x_2]$. Then 
$\frac{3}{2}\leq\rho(I)\leq\frac{8}{5}=1.6$.
\end{cor}






\begin{proof}
We get the lower bound $\frac{3}{2}\leq\rho(I)$ from the fact that $I^{(3)}\not\subseteq I^2$
\cite{refDST}, easily verified by Macaulay2.
For the upper bound we apply Lemma \ref{ADcompMeth}. 
Suppose that $\rho(I)>\frac{a}{b}$ for some positive integers $a$ and $b$. 
Then $\frac{a}{b}\leq\frac{m}{r}<\frac{6}{4}(1+\frac{j}{r})$ for some integers $m$ and $r$ such that
$I^{(m)}\not\subseteq I^r$. But if $\frac{a}{b}>\frac{6}{4}$, there are only finitely many $r$ for which this is possible.
(To see this write $\frac{a}{b}=\frac{6}{4}+\epsilon$ for some $\epsilon>0$.
Then we have $\frac{6}{4}+\epsilon\leq\frac{m}{r}<\frac{6}{4}(1+\frac{j}{r})\leq \frac{6}{4}(1+\frac{3}{r})$, 
but this is possible only for $r<\frac{9}{2\epsilon}$.)
For each $\frac{a}{b}>\frac{6}{4}$ 
one makes a list of all $m$ and $r$ with $\frac{a}{b}\leq\frac{m}{r}<\frac{6}{4}(1+\frac{j}{r})$
and tests whether $I^{(m)}\subseteq I^r$. With $a=8$ and $b=5$, the pairs $(r,m)$ that must be checked (ignoring
cases with $m/r\geq 2$ for which containment follows from \cite{refELS, refHoHu}, and ignoring cases
$(r,m')$ if $I^{(m)}\subseteq I^r$ for some $m<m'$) are:
(3, 5), (5, 8), (6, 10), (9, 15), (10, 16), (13, 21), (14, 23), (17, 28), (17, 29), (18, 29),
(21, 34), (25, 40), (29, 47) and (33, 53). Macaulay2 confirms containment for each of these cases,
so we obtain the bound $\rho(I)\leq\frac{8}{5}=1.6$.
\end{proof}

Now we consider the characteristic 3 counterexample given in \cite{refBCH}.
The following lemma is the same as before, with the same proof.

\begin{lem}\label{AgainHaHuProp3.5}
Let $\operatorname{char}(\field)=3$ and let 
$I=(x_0x_1(x_0^2-x_1^2), x_0x_2(x_0^2-x_2^2), x_1x_2(x_1^2-x_2^2), x_0(x_0^2-x_1^2)(x_0^2-x_2^2))
\subseteq\field[x_0,x_1,x_2]$. Then 
$I^{(3t)}=(I^{(3)})^t$ for all $t\geq 1$.
\end{lem}

However, now we have:

\begin{prop}\label{rhoProp}
Let $\operatorname{char}(\field)=3$ and let 
$I=(x_0x_1(x_0^2-x_1^2), x_0x_2(x_0^2-x_2^2), x_1x_2(x_1^2-x_2^2), x_0(x_0^2-x_1^2)(x_0^2-x_2^2))
\subseteq\field[x_0,x_1,x_2]$. Then $\rho(I)=\frac{5}{3}$.
\end{prop}

\begin{proof}
Waldschmidt's constant, defined as $\gamma(I)=\lim_{m\to\infty}\frac{\alpha(I^{(m)})}{m}$, is known to exist
(see \cite{refBH}). Using Lemma \ref{AgainHaHuProp3.5} we have 
$\gamma(I)=\lim_{3m\to\infty}\frac{\alpha(I^{(3m)})}{3m}=\lim_{3m\to\infty}\frac{\alpha((I^{(3)})^m)}{3m}=
\lim_{3m\to\infty}\frac{m\alpha(I^{(3)})}{3m}=\frac{\alpha(I^{(3)})}{3}$, which from Macaulay2 we compute
to be $\gamma(I)=3$. We also find from Macaulay2 that $\operatorname{reg}(I)=5$.
We now have the upper bound $\rho(I)\leq\frac{\operatorname{reg}(I)}{\gamma(I)}=\frac{5}{3}$ from \cite{refBH}.

To prove equality, we show that $I^{(15t)}\not\subseteq I^{9t+1}$ for all $t\geq1$.
We easily check that only the quintic generator of $I$ does not vanish at $q=[1:0:0]$.
Thus every homogeneous element of $(I^r)_i$ with $i<5r$ vanishes at $q$.
But there are 9 lines of $\pr2(\field)$ defined over the prime field ${\bf Z}/3{\bf Z}$ 
which do not contain $q$; let $F$ be 
the product of their defining forms. Note that exactly three of the linear factors
of $F$ vanish at each of the twelve ${\bf Z}/3{\bf Z}$-points of $\pr2(\field)$ other than $q$. 
Thus $F^{5t}\in (I^{(15t)})_{45t}$. Since $F^{5t}$ does not vanish
at $q$ and has degree less than $5r$ for $r=9t+1$, we have $I^{(15t)}\not\subseteq I^{9t+1}$.
Thus $\rho(I)\geq\frac{15t}{9t+1}$ for all $t\geq1$, and therefore $\rho(I)\geq\frac{15}{9}=\frac{5}{3}$.
\end{proof}

\section{Additional Counterexamples in $\pr 2$}\label{addcntrex}

In the case that $\field$ is the field of complex numbers, Proposition \ref{moreexs}
was observed by \cite{refDST}, with an explicit proof in the case that $j=3$.
Our proof runs along the same lines. 
Note that the hypothesis that $\field$ contains $j$
distinct roots of 1 is equivalent when $\operatorname{char}(\field)=p>0$
to $\field$ containing a finite field of order $p^s$ for some $s\geq 1$
with $p^s-1$ divisible by $j$.

\begin{prop}\label{moreexs}
Let $I=(x_0(x_1^j-x_2^j), x_1(x_0^j-x_2^j),x_2(x_0^j-x_1^j))\subseteq \field[x_0,x_1,x_2]$
for any $j\geq3$ and any field $\field$ of characteristic not equal to 2 containing $j$
distinct roots of 1. Then $I^{(3)}\not\subseteq I^2$. 
\end{prop}

\begin{proof}
The proof is similar to the one outlined in \cite{refDST}, which addresses the case $j=3$. 
Note that $I$ is the reduced ideal of a set of $j^2+3$ points in $\pr{2}$. More precisely, 
$V(I)$ is the union of two varieties $X$ and $Y$, where $X$ consists of all points having 
each coordinate equal to a $j^{th}$ root of 1 and $Y$ is the the set of points $[0:0:1], [0:1:0]$ and $[1:0:0]$. 
(To see this, note that $(x_0(x_1^j-x_2^j), x_1(x_0^j-x_2^j))$ vanishes only at the aforementioned
$j^2+3$ points, plus $j$ points on $x_0=0$ and $j$ points on $x_1=0$, but these last $2j$ points
do not lie in the zero locus of $x_2(x_0^j-x_1^j)$, whereas the other $j^2+3$ points do. 
Thus the radical of $I$ vanishes at exactly the claimed $j^2+3$ points. It is easy to check after localizing 
at any of the $j^2+3$ points that the three generators generate the localization of the ideal of the point.
Thus $I$ is radical and is thus the ideal of the $j^2+3$ points, as claimed.)

Set $F=(x_1^{j}-x_2^{j})(x_0^{j}-x_2^{j})(x_0^{j}-x_1^{j})$. We will show that 
$F\in I^{(3)}\setminus I^2$. Indeed, it is easy to see that $V(F)$ is the union of $3j$ 
lines such that exactly 3 of these lines pass through each point in $X$ and exactly $j$ 
of the lines pass through each point in $Y$. Since $j\geq 3$, we have that $F\in I^{(3)}$. Assume  
$$F=\sum_{i=1}^3 P_if_i^2+\sum_{1\leq i <j \leq 3}Q_{ij}f_if_j$$ 
where $f_1=x_0(x_1^j-x_2^j), f_2=x_1(x_0^j-x_2^j), f_3=x_2(x_0^j-x_1^j)$ are the generators of 
$I$ and $P_i,Q_{ij}$ are homogeneous polynomials  of degree $\deg(F)-(2j+2)=j-2$.  Going modulo  each 
of the variables $x_0$  and $x_1$ in the expression for $F$ displayed above, one obtains respectively
$$
\begin{matrix}
x_1^{j}x_2^{j}(x_1^{j}-x_2^{j}) &= &x_1^2x_2^{2j}\overline{P}_2+x_2^2x_1^{2j}\overline{P}_3+x_1^{j+1}x_2^{j+1}\overline{Q}_{23} \\
-x_0^{j}x_2^{j}(x_0^{j}-x_2^{j}) & =& x_0^2x_2^{2j}\widetilde{P}_1+x_2^2x_0^{2j}\widetilde{P}_3-x_0^{j+1}x_2^{j+1}\widetilde{Q}_{13},\\
\end{matrix}
$$
where the bar denotes residue classes modulo $(x_0)$ and $\widetilde{\ }$ denotes residue classes modulo $(x_1)$.

Comparing the degree of $x_2$ in each term of the first equation, one sees that the monomial 
$x_1^{2j}x_2^{j}$ on the left  must arise as the product of $x_2^2x_1^{2j}$ and a term of the 
form $x_2^{j-1}$ appearing in $\overline{P}_3$, hence also in $P_3$. Hence the coefficient 
of the monomial  $x_2^{j-1}$ in the expression of $P_3$ must be $1$. Similarly, the monomial 
$-x_0^{2j}x_2^{j}$ in the second equation forces $P_3$ to contain a term of the form $-x_2^{j-2}$.  
Since the characteristic of the field $\field$ is not two, this yields a contradiction, as the coefficient 
of the monomial $x_2^{j-1}$ in $P_3$ cannot simultaneously be $1$ and $-1$.
\end{proof}

We now obtain positive characteristic counterexamples to $I^{(2r-1)}\subseteq I^r$ for 
infinitely many $r$ for each characteristic $p>2$.

\begin{prop}\label{counterexP2charp}
Let $\operatorname{char}(\field)=p>2$. Let $\field'\subseteq \field$ be a finite field of order $p^s$ for some $s\geq 1$.
Let $I$ be the ideal of all of the $\field'$-points of $\pr2(\field)$ but one, which we may assume is $q=[0:1:0]$.
If $r=(p^s+1)/2$ (and hence $2r-1=p^s$), then $I^{(2r-1)}\not\subseteq I^r$. 
\end{prop}

\begin{proof}
Let $J$ be the ideal of all $\field'$-points
of $\pr2(\field)$. First, we show that $\alpha(I)=\alpha(J) = p^s+1$. 
If $B\in I$ is homogeneous of degree $p^s$, every line in $\pr2(\field')$ not 
through $q$ meets $V(B)$ in $p^s+1$ points and thus B\'ezout's Theorem implies that 
every one of the $p^{2s}$ such lines is a factor of $B$.
This contradicts the assumption that the degree of $B$ is $p^s$, thus 
$\alpha(I)\geq p^s+1$ and, since $J\subset I$, also $\alpha(J)\geq p^s+1$.
Moreover, if $D$ is the product of all of the $p^s+1$ lines through any given point of $\pr2(\field')$, 
then $D$ vanishes at all points of $\pr2(\field')$ so $D\in J\subset I$. Hence
$\alpha(I) = \alpha(J) = p^s+1$.

Next we claim that 
$q$ is in the zero locus of $I_t$ for degrees $t < 2p^s-1$. To prove this, note
that I defines a point scheme X, which is the union of
two complete intersections: the $p^{2s}$ points $U= \pr2(\field')\setminus V(x_2)$ not at infinity
and the $p^s$ points $W=(\pr2(\field')\cap V(x_2))\setminus\{q\}$ of $\pr2(\field')$ at
infinity other than $q$. We have $I_U=(G,H)$
(where $G=\Pi_{\zeta\in\field'}(x_1-\zeta x_2)$ is the product of the $p^s$ forms other than $x_2$ defining lines
through $[1:0:0]$ and $H=\Pi_{\zeta\in \field'}(x_0-\zeta x_2)$ is the 
product of the $p^s$ forms other than $x_2$ defining lines through $[0:1:0]$)
and $I_W=(x_2,F)$ (where $F=\Pi_{\zeta\in\field'}(x_1-\zeta x_0)$ is the product of 
the $p^s$ forms other than $x_0$ defining lines through $[0:0:1]$). Since $I = I_U \cap I_W$,
we have that $A\in I$ if and only if $A = gG+hH = ax_2+fF$ for some $a,f,g,h\in k[\pr2(\field)]=\field[x_0,x_1,x_2]$. 

Suppose that $A$ does not vanish at $q$. Thus $x_2$ does not divide $A$, hence
$A$ and therefore $f$ are not 0 modulo $x_2$.
Modulo $x_2$, the relation $A = gG+hH = ax_2+fF$
yields  $\overline{g}\overline{G}+\overline{h}\overline{H}=\overline{f}\overline{F}$, where the bar 
denotes images in $\field[\pr{2}]/(x_2)=\field[x_0,x_1]$.
In the quotient ring $\field[x_0,x_1]$ one has that 
$\overline{F}=\Pi_{\zeta\in\field'}(x_1-\zeta x_0)=x_1(x_1^{p^s-1}-x_0^{p^s-1})$,
$\overline{G}=x_1^{p^s}$ and $\overline{H}=x_0^{p^s}$. 
Any solution $\overline{f},\overline{g},\overline{h}$ of $\overline{g}\overline{G}+\overline{h}\overline{H}=\overline{f}\overline{F}$ satisfies 
$x_1^{p^s}(\overline{g}-\overline{f})+x_0^{p^s-1}(x_0\overline{h}+x_1\overline{f})=0$,  
and since the  module of syzygies on $x_0^{p^s-1},x_1^{p^s}$ is generated by the Koszul syzygy, this 
implies that the degree of $f$ is at least $p^s-1$ (unless 
$\overline{g}-\overline{f}=0$ and $x_0\overline{h}+x_1\overline{f}=0$, in which case
$x_0\overline{h}+x_1\overline{f}=0$ implies $x_0$ divides $\overline{f}=\overline{g}=x_1^{p^s}$, which is impossible).
Thus the least degree of an element of $I$ not vanishing at $q$ is at least $2p^s-1$. But one 
such element in $\field[\pr{2}]$ is $FG/x_1$, and it has degree exactly $2p^s-1$, 
so the least degree of an element of $I$ not vanishing at $q$ is exactly $2p^s-1$.

To show that $I^r$ does not contain $I^{(p^s)}$ for $r=(p^s+1)/2$ (where 
$p$ is odd, hence $p > 2$) consider the product $P$ of the $\field'$-lines in $\pr{2}(\field')$ 
not containing $q$. This has degree $p^{2s}$, and there are $p^s$ such lines through 
each point of $X$, so $P$ is in $I^{(p^s)}$. From the previous considerations, the least degree of a form 
in $I^r$ not vanishing at $q$ is $r(2p^s-1)=(p^s+1)(2p^s-1)/2=p^{2s}+(p^s-1)/2 > p^{2s}$. 
Thus $P$ is not in $I^r$ and hence $I^r$ does not contain $I^{(p^s)}=I^{(2r-1)}$.
\end{proof}

\section{Counterexamples in $\pr N$, $N\geq2$}\label{N>=2}

Throughout this section, let $\field$ be a field of characteristic $p>0$ and let $\field'$ be the subfield of order $p$.
Let $I\subseteq \field[\pr{N}]=\field[x_0,\ldots, x_N]$ be the ideal of 
all of the $\field'$-points of $\pr{N}=\pr{N}(\field)$ but one.
In this section we prove that 
$I^{(Nr-(N-1))} \not\subseteq I^r$ always holds for the following cases:
\begin{enumerate}
\item
$p>2$, $r=2$ and $N=(p+1)/2$ (in which case $Nr-(N-1)=(p+3)/2$) and 
\item
$r=(p+N-1)/N$ (in which case $Nr-(N-1)=p$), $p>(N-1)^2$ and $p\equiv 1 (\!\!\!\mod N)$.
\end{enumerate}

\begin{rem}
Proposition \ref{counterexP2charp} shows that the second statement is
true in the case of $N=2$ for arbitrary primes $p>2$ in a more general setting, 
where $\field'$ is a finite field of order $p^s$.
We also note that $I^{(Nr-(N-1))} \subseteq I^r$ when $N=4$, $r=2$ and $p=5$,
so some condition (such as our $p>(N-1)^2$) is needed to ensure that $p$ is not too small.
\end{rem}

We begin by establishing a few preliminary lemmas.

\begin{lem}\label{binomialidentitygeneral}
Let $Q(x_0,\ldots,x_N)=((x_0+1)\cdots(x_N+1))^{p-1}\in \field'[x_0,\ldots,x_N]$, where $N$ is a positive integer
with $N\leq p$. Let $Q_{p-1,N}$ be the homogeneous component of $Q$ in degree $p-1$. 
Then $Q_{p-1,N}$ vanishes to order exactly $p-N$ at $[1:\cdots:1]$.
\end{lem}

 \begin{proof} 
 We suppress the subscript $N$ in $Q_{p-1,N}$, since doing so will not cause confusion in this proof.
To translate $[1:\cdots:1]$ to $[1:0:\cdots:0]$, we make the substitution 
$Q_{p-1}(x_0,x_1+1,\ldots,x_N+1)$; we must show that
$Q_{p-1}(1,x_1+1,\ldots,x_N+1)$ has no monomial terms of degree less than $p-N$
but has a term of degree $p-N$. For appropriate non-negative
integers $\alpha_i$ and elements $c_{\alpha_1,\ldots,\alpha_N}\in\field'$, we can write
$Q_{p-1}(1,x_1+1,\ldots,x_N+1)=\sum_{\alpha_1+\cdots+\alpha_N\leq p-1}
c_{\alpha_1,\ldots,\alpha_N}x_1^{\alpha_1}\cdots x_N^{\alpha_N}$. With
$s=\alpha_1+\cdots+\alpha_N \leq p-1$, we have 
\begin{eqnarray*}
\alpha_1!\cdots\alpha_N!c_{\alpha_1,\ldots,\alpha_N}=&
\Big(\frac{\partial^s}{\partial x_1^{\alpha_1}\cdots\partial x_N^{\alpha_N}}Q_{p-1}(x_0,x_1+1\ldots,x_N+1)\Big)\Big|_{(x_0,\ldots,x_N)=(1,0,\ldots,0)}\\
=&\Big(\frac{\partial^s}{\partial x_1^{\alpha_1}\cdots\partial x_N^{\alpha_N}}Q_{p-1}(x_0,\ldots,x_N)\Big)\Big|_{(x_0,\ldots,x_N)=(1,1,\ldots,1)}
\end{eqnarray*}
and since $a_i\leq s < p$, we see that $c_{\alpha_1,\ldots,\alpha_N}=0$ if and only if 
$\Big(\frac{\partial^s}{\partial x_1^{\alpha_1}\cdots\partial x_N^{\alpha_N}}Q_{p-1}\Big)(1,1,\ldots,1)=0$. 
We must show that $c_{\alpha_1,\ldots,\alpha_N}=0$ whenever $s<p-N$ but is sometimes nonzero when $s=p-N$.
For notational simplicity, let 
$$Q_{p-1}^{\alpha_1,\ldots,\alpha_N}(x_0,\ldots,x_N)=\frac{\partial^s}{\partial x_1^{\alpha_1}\cdots\partial x_N^{\alpha_N}}Q_{p-1}(x_0,\ldots,x_N).$$
Also, given a polynomial $F$, we will denote the homogeneous component of $F$ of degree $i$ by $F_i$.
Then we have
\begin{eqnarray*}
Q_{p-1}^{\alpha_1,\ldots,\alpha_N}(x_0,\ldots,x_N) =& \frac{\partial^s}{\partial x_1^{\alpha_1}\cdots\partial x_N^{\alpha_N}}Q_{p-1}(x_0,\ldots,x_N)=\left(\frac{\partial^s}{\partial x_1^{\alpha_1}\cdots\partial x_N^{\alpha_N}}Q(x_0,\ldots,x_N)\right)_{p-1-s} \\
=& \left((x_0+1)^{p-1}\prod_{i=1}^N\frac{\partial^{\alpha_i}}{\partial x_i^{\alpha_i}}((x_i+1)^{p-1})\right)_{p-1-s}\\
=&  \left((x_0+1)^{p-1}\prod_{i=1}^N\alpha_i!  \binom{p-1}{\alpha_i} (x_i+1)^{p-1-\alpha_i}\right)_{p-1-s}.\\
\end{eqnarray*}
Setting $x_i=x$ for all $i$ gives
$$Q_{p-1}^{\alpha_1,\ldots,\alpha_N}(x,\ldots,x)=
\left((x+1)^{p-1}\prod_{i=1}^N\alpha_i! \binom{p-1}{\alpha_i} (x+1)^{N(p-1)-s}\right)_{p-1-s},$$
and evaluating at $x=1$ gives
$$Q_{p-1}^{\alpha_1,\ldots,\alpha_N}(1,\ldots,1)=
\left(\prod_{i=1}^N\alpha_i!  \binom{p-1}{\alpha_i}\right)\binom{(N+1)(p-1)-s}{p-1-s}.$$
This is zero if and only if 
$$\binom{(N+1)(p-1)-s}{p-1-s}=\frac{((N+1)(p-1)-s)\cdots((N+1)(p-1)-s-(p-1-s)+1)}{(p-1-s)!}$$
is zero, and since the denominator of the fraction has no factors of $p$, this is zero
if and only if the numerator has a factor of $p$. But since the width of the interval
$[(N+1)(p-1)-s-(p-1-s)+1,(N+1)(p-1)-s]$ is $p-2-s$, there is at most one factor divisible by $p$.
We now show that this factor, when it exists, is $Np$. Now, $Np$ is a factor if and only if
$N(p-1)=(N+1)(p-1)-s-(p-1-s)<Np\leq (N+1)(p-1)-s$, but $N(p-1)<Np$ always holds, as does
$(N-1)p\leq N(p-1)$ since $N\leq p$ by hypothesis. Thus there is a factor divisible by $p$
if and only if $Np\leq (N+1)(p-1)-s$, but this is equivalent to $s < p-N$. I.e., 
$c_{\alpha_1,\ldots,\alpha_N}=0$ if and only if $s < p-N$, as required.
\end{proof}

\begin{rem}\label{p-N}
Let $T=(x_1-x_0,\ldots,x_N-x_0)$ denote the ideal of the point $[1:\cdots:1]\in \field[\pr{N}]$. 
Lemma \ref{binomialidentitygeneral} shows under the given hypotheses that $Q_{p-1,N}\in T^{\,p-N}$.
\end{rem}

\begin{lem}\label{symbQ}
Let $Q_{p-1,N}$ be the homogeneous polynomial defined in Lemma \ref{binomialidentitygeneral}. 
Let $\pi$ be a point in $\pr{N}(\field)$ with defining ideal 
$I_\pi$ and denote by $t$ the number of coordinates of $\pi$ which are equal to zero.
Then $Q_{p-1,N}(x_0^{p-1},\ldots,x_N^{p-1})\in I_\pi^{p-N+t}$ if $p - N +t \geq 0$ and $t<N$.
\end{lem}

\begin{proof}
If $t=0$ the conclusion follows from Remark  \ref{p-N} as 
$$Q_{p-1,N}(x_0^{p-1},\ldots,x_N^{p-1})\in (x_1^{p-1}-x_0^{p-1},\ldots,x_N^{p-1}-x_1^{p-1})^{p-N}\subset I_\pi^{\ p-N}.$$ 
For other values of $t$ ($1\leq t < N$), we project to a smaller projective space where we can use the case $t=0$.  
Because  $Q_{p-1,N}$ is symmetric, we may assume that the last $t$ coordinates of $\pi$ are zero. 
Denote by $\overline{\pi}$ the point $\pi$ viewed in the $(N-t)$-dimensional projective space spanned 
by its non-zero coordinates and by $\overline{I}_{\overline{\pi}}$ the corresponding ideal of this point in the 
smaller projective space (and by notational abuse also its extension to $\field[\pr{N}]$). Then the claim 
that  $Q_{p-1}(x_0^{p-1},\ldots,x_N^{p-1})\in I_\pi^{p-N+t}$ is equivalent to 
$$Q_{p-1}(x_0^{p-1},\ldots,x_N^{p-1})\in (\overline{I}_{\overline{\pi}}+(x_{N-t+1},\ldots,x_N))^{p-N+t}.$$

It is easy to see that $Q_{p-1}(x_0^{p-1},\ldots,x_N^{p-1})$ can be written 
as a sum of $Q_{p-1,N-t}(x_0^{p-1},\ldots,x_{N-t}^{p-1})$
together with terms each of which is divisible by at least one of the pure powers $x_{N-t+1}^{p-1},\ldots,x_N^{p-1}$. 
Clearly these latter terms are contained in 
$(x_{N-t+1},\ldots,x_N)^{p-1}$ and by Remark \ref{p-N} the latter polynomial is contained in 
$\overline{I}_{\overline{\pi}}^{\ p-(N-t)}$. Since 
$$(x_{N-t+1},\ldots,x_N)^{p-1}+\overline{I}_{\overline{\pi}}^{\ p-(N-t)}\subset (\overline{I}_{\overline{\pi}}+(x_{N-t+1},\ldots,x_N))^{p-N+t},$$ 
the claim is now proved.
\end{proof}

\begin{lem}\label{Ggen}
Assume the prime $p$ is odd. Set $N=\frac{p+1}{2}$, let $q$ be a $K'$-point of $\pr{N}(\field)$ and let
$I$ be the ideal of all of the $K'$-points of $\pr{N}(\field)$ but $q$. 
Then  $I^{(\frac{p+3}{2})}$ contains a form of degree $p^2$ which does not vanish at $q$.
\end{lem}

\begin{proof}
After a change of coordinates, we may assume that 
$q=V(x_1,\ldots,x_N)=[1:0:\cdots:0]$.  
We claim that we obtain a form as in the statement of this lemma by setting
$$G=x_0(x_0^{p-1}-x_N^{p-1})(Nx_0^{p-1}+\sum_{i=1}^{N-1}x_i^{p-1})Q_{p-1,N}(x_0^{p-1},\ldots,x_N^{p-1}),$$
where $Q_{p-1,N}$ is the homogeneous polynomial defined in Lemma \ref{binomialidentitygeneral}. 
Indeed, $\deg(G)=1+2(p-1)+(p-1)^2=p^2$. Furthermore it is easy to see $G$ 
does not vanish at $q$ as each of the factors contains a pure power of $x_0$. 

To check $G\in I^{(\frac{p+3}{2})}$, let $\pi$ denote a point in $V(I)$ 
(hence $\pi\neq q$) and let $t$ denote the number of coordinates of $\pi$ that are zero.
Say $t=N$. If $\pi\neq [0:\cdots:0:1]$, then 
$x_0$ has order of vanishing 1 at $\pi$ and
$x_0^{p-1}-x_N^{p-1}$ has order of vanishing $p-1$ at $\pi$, 
so $G$ has order of vanishing at least $p$ at $\pi$ (in fact exactly $p$, but we do not need this).
Since $p$ is odd, we have $p\geq \frac{p+3}{2}$, so $G\in I_\pi^p\subseteq I_\pi^{\frac{p+3}{2}}$.
If $\pi= [0:\cdots:0:1]$, then 
$x_0$ has order of vanishing 1 at $\pi$
and $Nx_0^{p-1}+\sum_{i=1}^{N-1}x_i^{p-1}$ has order of vanishing $p-1$ at $\pi$, 
so as before $G\in I_\pi^{\frac{p+3}{2}}$.

Now say $0\leq t<N$. We recall by Lemma \ref{symbQ} that
$Q_{p-1,N}(x_0^{p-1},\ldots,x_N^{p-1})\in I_\pi^{p-N+t}= I_\pi^{\frac{p-1}{2}+t}$. 
The conclusion now follows immediately from this relation when $t\geq 2$.
To establish that $G\in I^{(\frac{p+3}{2})}$ for $t=0$ or $t=1$,
it is enough to show 
$$x_0(x_0^{p-1}-x_N^{p-1})(Nx_0^{p-1}+\sum_{i=1}^{N-1}x_i^{p-1})$$ 
is an element of  $I_\pi^2$ if $t=0$ or of $I_\pi$ if $t=1$. 

If $t=0$, then it is easy to see that each of the factors $x_0^{p-1}-x_N^{p-1}$ 
and $Nx_0^{p-1}+\sum_{i=1}^{N-1}x_i^{p-1}$ are in $I_\pi$, hence the 
product is in $I_\pi^2$. If $t=1$, then either the $x_0$ coordinate of $\pi$ is 
zero, in which case $x_0\in I_\pi$, or the $x_N$ coordinate of $\pi$ is zero, in which 
case $Nx_0^{p-1}+\sum_{i=1}^{N-1}x_i^{p-1}\in I_\pi$, or neither one of the 
coordinate functions $x_0,x_N$ vanish at $\pi$, in which case $x_0^{p-1}-x_N^{p-1}\in I_\pi$.  
In any case we obtain the desired conclusion, hence $G\in I^{(\frac{p+3}{2})}$.
\end{proof}

\begin{lem}\label{Fgen}
Let  $I$ be the ideal of all of the $K'$-points of $\pr{N}(\field)$ 
but one, which we denote $q$. Then  $I^{(p)}$ contains a form of degree $p(p+N-2)$ which does 
not vanish at $q$.
\end{lem}

\begin{proof}
After a projective change of coordinates, we may assume that 
$q=V(x_1,\ldots,x_N)=[1:0:\cdots:0]$.  
We claim that we obtain a form as in the statement of this lemma by setting
$$F=x_0^{N-1}\prod_{i=1}^N(x_0^{p-1}-x_i^{p-1})Q_{p-1,N}(x_0^{p-1},\ldots,x_N^{p-1}),$$
where $Q_{p-1,N}$ is the homogeneous polynomial defined in 
Lemma \ref{binomialidentitygeneral}. Indeed, $\deg(F)=N-1+N(p-1)+(p-1)^2=p(p+N-2)$. Furthermore it is easy 
to see $F$ does not vanish at $q$ as each of the factors contains a pure power of $x_0$.

To check $F\in I^{(p)}$, let $\pi\neq q$ denote a point in 
$V(I)$ and let $t$ be the number of coordinates of $\pi$ that are zero. 
Say $t=N$. Thus $x_0=0$ at $\pi$, so $F$ vanishes to order at least $(N-1)+(N-1)(p-1)=(N-1)p$ at $\pi$.
Now say $t<N$.
By Lemma \ref{symbQ}, $Q_{p-1,N}(x_0^{p-1},\ldots,x_N^{p-1})\in I_\pi^{p-N+t}$.
To conclude that $F\in I_\pi^p$ we need now only show that the product 
$$x_0^{N-1}\prod_{i=1}^N(x_0^{p-1}-x_i^{p-1})$$ 
of its other factors vanishes to order at least $N-t$ at 
$\pi$. Indeed, if the $x_0$ coordinate of $\pi$ is non-zero, each of the factors 
$x_0^{p-1}-x_i^{p-1}$ vanishes at $\pi$ if and only if the corresponding 
coordinate $x_i$ is non-zero so the product vanishes exactly $N-t$ times. 
If the $x_0$ coordinate is zero, then clearly the product vanishes at least 
$2N-t=(N-1)+(N-(t-1))\geq N-t$ times. 
Thus we have $F\in I_\pi^{\ p}$ for every $\pi\neq q$, hence $F\in I^{(p)}$.
\end{proof}

As a second ingredient for the proof of our main theorems, we require a computation 
of the  degrees of elements of $I$ not vanishing at the excluded point, for 
comparison with the results in Lemmas \ref{Ggen} and \ref{Fgen}. In order to do this, we use a tool 
from linkage theory, namely basic double linkage, details on which can be found 
in \cite{refKMMNP}. The  idea of basic double linkage is to construct an ideal as 
$I = FB+J$, where $J \subseteq B$ are homogeneous ideals in a polynomial ring 
$R$ and $F$ is a homogeneous element of $R$, such that $B$ is saturated of codimension 
$c$ (in our case $c=\dim R-1$), $J$ has codimension $c-1$ and is arithmetically 
Cohen-Macaulay (ACM), and $F$ is a non-zerodivisor on $R/J$. Then  
\cite[Proposition 5.1]{refKMMNP} gives that $I$ is saturated and that $I$ is  ACM if 
and only if $B$ is ACM. 
(Although not relevant here, it is instructive to note that in this setup the ideal $I$ can be 
linked in two steps to $G$ via Gorenstein ideals, which justifies the terminology given to this 
construction.)

\begin{lem}\label{BDL}
Let $H$ be a $K'$-hyperplane in $\pr{N}(K)$ not containing a given $K'$-point $q$. 
After a $K'$-projective change of coordinates, we may assume that
$H$ is defined by $F=x_0=0$, that $q=[1:0\cdots:0]$ and that
$\field[\pr{N}]/(F)=\field[x_1,\ldots,x_n]\subset\field[x_0,\ldots,x_N]$.
Let $J$ denote the ideal of $R=\field[\pr{N}]$ generated by all the forms
in $\field[\pr{N}]/(F)$ vanishing on every $\field'$-point in 
$H$. Let  $B$ be the ideal of the $p^{N}-1$ points in $\pr{N}(\field')$ which
are not on $H$ and are distinct from $q$. 
Let $I$ be the basic double link of $B$ obtained as $I=FB+J$. 
Then $I$ is the radical ideal of all of the $\field'$-points of $\pr{N}(\field)$ other than $q$.
\end{lem}

\begin{proof}
First note that the setting of the basic double link construction is satisfied. As $B$ is an ideal 
of a finite set of points and $J$ is the ideal of the projective cone over a finite set of points, 
they are ACM, furthermore $J \subseteq B$, since $V(J)$ is a union of lines in $\pr{N}(K)$, 
one for each $\field'$-point of $H$, which cover all $\field'$ points in $\pr{N}(\field)$ and 
$V(B)$ is thus contained in $V(J)$. Additionally, $F$ is clearly a non-zerodivisor on $R/J$. 
Then \cite[Proposition 5.12]{refKMMNP} gives that $I$ is saturated. It is clear that 
$I$ does not vanish at $q$ since neither $B$ nor $F$ do. Moreover, $I$ vanishes 
at all other $\field'$-points since $J$ and also one of $F$ or $B$ do. Let $A$ be the ideal
of all of the $\field'$-points of $\pr{N}(\field)$ other than $q$. Since $A$ and $I$ are both 
saturated and have the same zero-locus while $A$ is radical, we see $I\subseteq A$.
It is now enough to check that their associated sheaves have the same stalks to conclude
that $I=A$. But off $x_0=0$, $A$ and $I$ both have the same stalks as $B$,
and, away from $V(B)$, $A$ and $I$ both have the same stalks as $(F)+J$
(note that $(F)+J$ is radical since $H$ is transverse to the lines defined by $J$).
\end{proof}

In the following,  for a point scheme $X$ in $\pr{N}$  we define the {\em h-vector} of 
$X$ (or of $I_X\subset R=K[\pr{N}]$) to be the first difference of the Hilbert function of 
$R/I_X$. Then the degree of $X$ is given by the sum of the entries of its h-vector, 
which we denote by $h_i(I_X)$. We remark that if the residue field of $R$ is infinite, 
the $h$-vector of $X$ can be obtained by the process of artinan reduction as the Hilbert 
function of $R/(I_X+L)$, where $L$ is a linear form regular on $R/I_X$. 

\begin{prop}\label{bpfdegreegen}
Let $I$ be the ideal of all of the $\field'$-points of $\pr{N}(\field)$ but one.
Then  the smallest degree $n$ such that $I_n$ contains a form which does not 
vanish at every point of $\pr{N}(\field)$ is $n=N(p-1)+1$.
\end{prop}

\begin{proof}
Let $q$ and $H$ be as in Lemma \ref{BDL}, and let
$C$ be the defining ideal of all $K'$-points away from $H$. 
These just comprise the grid of $p^N$ points in affine $N$ space with coordinates in $K'$.
Thus $C=\cap (x_1-a_1x_0,\ldots,x_N-a_Nx_0)$, where the intersection is over
all $(a_1,\ldots,a_N)\in(K')^N$. In fact, $C$ is equal to the complete intersection 
$C'=(x_1(x_1^{p-1}-x_0^{p-1}),\ldots,x_N(x_N^{p-1}-x_0^{p-1}))$. It is easy to see that
$C'\subseteq C$, that both are radical ideals and both have the same zero locus.
Since $C$ is saturated by construction and $C'$ is saturated by the unmixedness theorem,
we have $C=C'$. Since, as is well known, the Koszul complex gives a minimal free resolution 
for a complete intersection, the regularity of $C'$ is $\Sigma-N+1$, where $\Sigma=Np$
is the sum of the degrees of the $N$ generators. Thus by the Cayley-Bacharach Theorem
(see the remark after Therem 3 of \cite{refDGO}), the least degree of a form $G$
vanishing on all points of $C$ except $q$ is $Np-N=N(p-1)$.

Now we turn our attention to $I$. By Lemma \ref{BDL}, this ideal is the basic 
double link  $I=FB+J$.  The generators of $J$ all vanish at $q$ and among the 
minimal generators of $FB$ only $FG$ does not vanish at $q$, this means that 
any form in $I$ not vanishing at $q$ arises as a combination of the minimal 
generators which uses $FG$ nontrivially (with non-zero coefficient). In particular, 
this shows both that there is a form, namely $FG$, of degree $N(p-1)+1$ in $I$ 
not vanishing at $q$ and also that there are no forms of smaller degree with this property.
\end{proof}

\begin{thm}\label{reg=2}
Let $p$ be odd.
Set $N=\frac{p+1}{2}$ and let $I$ be the ideal of all of the $\field'$-points of $\pr{N}(\field)$ but one.
Then we have $I^{(\frac{p+3}{2})} \not\subseteq I^2$. 
\end{thm}

\begin{proof}
 Consider the form $G\in I^{(\frac{p+3}{2})}$ given by Lemma \ref{Ggen}. Proposition \ref{bpfdegreegen}
implies that any element of $I^2$ not vanishing at $q$ must have degree at least $2(N(p-1)+1)=p^2+1$.  
Since $G$ is a form of degree $p^2$ not vanishing at $q$, it follows that $G\not \in I^2$. 
This concludes the proof.
\end{proof}

\begin{thm}\label{symb=p}
Let  $I$ be the ideal of all of the $\field'$-points of $\pr{N}(\field)$ but one. 
Let $r=(p+N-1)/N$ and assume that $p\equiv 1 (\!\!\!\mod N)$ and $p > (N-1)^2$. 
Then we have $I^{(p)}=I^{(Nr-(N-1))} \not\subseteq I^r=I^{(p+N-1)/N}$. 
\end{thm}

\begin{proof}
Consider the form $F\in I^{(p)}$ given by Lemma \ref{Fgen}. We shall prove that $F\not \in I^{(p+N-1)/N}$. 
Assume towards a contradiction that  $F \in I^{(p+N-1)/N}$. By Proposition \ref{bpfdegreegen}, 
since $F$ does not vanish at $q$, it would have to be the case that $F\in (I_{\geq N(p-1)+1})^{(p+N-1)/N}$, yielding
\begin{eqnarray*}
\deg(F)& \geq& \frac{(N(p-1)+1)(p+N-1)}{N}=(p-1)(p+N-1) +\frac{p+N-1}{N} \\
&=&p(p+N-2)+2-N+\frac{p-1}{N}>p(p+N-2),
\end{eqnarray*}
where the last inequality follows because $p>(N-1)^2$ implies 
$2-N+\frac{p-1}{N}>0$. This contradicts $\deg(F)=p(p+N-2)$, and so concludes the proof.
\end{proof}

\section*{Acknowledgements}
We  thank I.\ Dolgachev, E.\ Guardo and A.\ Van Tuyl for helpful comments,
and M.\ Dumnicki, T.\ Szemberg and H.\ Tutaj-Gasi\'nska for sharing the latest 
version of their paper with us.

\end{document}